\documentclass[11pt]{article}
\usepackage{verbatim}
\usepackage{amssymb,amsmath,amsthm,bm,bbm, amsfonts, xspace, url}
\usepackage[capitalise]{cleveref}
\makeatletter
\if@cref@capitalise
\usepackage{graphicx,xcolor}
\usepackage{epsfig}
\newtheorem{theorem}{Theorem}
\newtheorem{corollary}{Corollary}[section]
\newtheorem{lemma}[corollary]{Lemma}
\newtheorem{proposition}[corollary]{Proposition}
\newtheorem{conjecture}[corollary]{Conjecture}

\newtheorem{remark}{Remark}
\newcommand{\Prob} {{\mathbb P}}
 
\newcommand{\N}{{\mathbb N}}

\newcommand{\R}{{\mathbb{R}}}
\newcommand{\C}{{\mathbb C}}

\newcommand{\Q}{{\mathbb Q}}

\def \Half {{\mathbb H}}

\def \loop {{\cal L}}

\def \H {{ \mathcal H}}

\newcommand{\restr}[1]{|_{#1}}
\newcommand\nth{\textsuperscript{th}\xspace}
\newcommand\st{\textsuperscript{st}\xspace}

\newcommand{{\pe}}  {\partial_e}
\newcommand {{\lodd}} {{\mathcal J}}
\newcommand{{\inrad}} {{\rm inrad}}

\newcommand {{\cent}} {{\bf c}}
\newcommand {{\eb}}  {{\bf e}}
\newcommand{{\osc}} {{\rm osc}}
\newcommand{{\dyadic}}  {{\mathcal Q}}

\newcommand{{\partition}} {{\mathcal P}}
\newcommand {{\bloop}}  {{\mathcal O}}
\newcommand {{\crad}}  {{\rm crad}}
\newcommand {{\wind}} {{\rm wind}}
\newcommand {{\z}} {{\bf z}}
\newcommand {{\whoknows}} {{\mathcal A}}
\newcommand{{\exc}}{{\mathcal E}}
\newcommand{{\loops}} {{\loop}}

\newcommand {{\squares}}{{\mathbb A}}

\newcommand {{\lattice}}  {{\mathcal L}}

\author{Stephen Yearwood}
\title{The topology of $SLE_{\kappa}$ is random for $\kappa >4$}

\begin{document}
\maketitle

\begin{abstract}
    We study the topology of $SLE$ curves for $\kappa > 4$. More precisely, we show that, a.s., there is no homeomorphism $\Phi: \overline{\Half} \rightarrow \overline{\Half}$, taking the range of one independent $SLE$ curve to another for $\kappa \in (4,8)$. Furthermore, we extend the result to $\kappa \geq 8$ by showing that there is no homeomorphism taking one $SLE$ curve to another, when viewed as curves modulo parametrization.
\end{abstract}
\section{Introduction}
\subsection{Initial Overview}
The Schramm-Loewner Evolutions ($SLE_{\kappa}$), introduced by Oded Schramm \cite{schramm0}, describes a family of probability distributions,
parameterized by $\kappa > 0$ on non-self traversing curves connecting two boundary points in a planar, simply connected domain. They are characterized by a conformal invariance condition and a domain
Markov property. They were initially observed as possible candidates for the scaling limits of various discrete lattice models in statistical physics; we now know that some of these convergences do hold, and so $SLE$ exhibits a universality in its definition. 

$SLE_{\kappa}$ is the random growth of a set $K_t$, as described through a conformal map $g_t(z)$ on the complement of $K_t$. This map $g_t(z)$ is the solution of the Loewner differential equation driven by a Brownian motion, whose `speed' is determined by a single parameter $\kappa$. Rhodes and Schramm in \cite{schramm-sle} showed that for $\kappa \neq 8$, a.s. there is a (unique) continuous path $\eta :[0, \infty) \to \overline{\Half}$ such that for each $t > 0$ the set $K_t$ is the union of $\eta[0, t]$ and the bounded connected components of $\Half \setminus \eta[0, t]$. This has also been shown for $\kappa =8$, but was  dealt with separately. We call the path $\eta$ the SLE trace or $SLE$ curve. It was shown as well in \cite{schramm-sle} that $\underset{t \to \infty}{\text{lim}}|\eta(t)| = \infty$ a.s. We will need the following facts about the curve which are proven in \cite{schramm-sle}:
\begin{itemize}
\item If $\kappa \leq 4$, then $\eta$ is simple with
$\eta(0,\infty) \subset \Half$. 
\item  If $ 4 < \kappa < 8$, then $\eta(0,\infty)$ has double
points and intersects $\R$. 
\item If $\kappa \geq 8$, the curve is space-filling.

\end{itemize}
There are three (well studied) variants of $SLE$ : chordal $SLE$, which connects two boundary points (prime ends) in a given domain, radial $SLE$, which connects a boundary point to an interior point, and whole-plane $SLE$, which connects two points on the Riemann sphere. In this paper, we will focus on chordal $SLE$, but we expect that our results can be generalized to other cases. One can also see \cite{lawler-book, bn-sle-notes, werner-notes} for some expository work on $SLE$ which go into details beyond the scope of this paper.

 
 Most works on $SLE$ have focused on its geometric and probabilistic properties, e.g., Hausdorff dimensions of various subsets of the curves, formulas for the probabilities of various events, and connections to other random objects. In this work, we will address a very basic question about the topology of $SLE$: namely, is the topology of the curve deterministic? Said differently, if we have two independent chordal SLE$_\kappa$ curves $\eta^1$ and $\eta^2$ (viewed as curves modulo time parametrization), does there a.s.\ exist a homeomorphism $  \overline{\mathbb H} \rightarrow \overline{\mathbb H}$ taking $\eta^1$ to $\eta^2$?

Since $SLE_\kappa$ is a simple curve for $\kappa \leq 4$, the answer to the above question is clearly affirmative in this case. For $\kappa > 4$, however, the answer is less obvious. On the one hand, many events for SLE$_\kappa$ occur with probability strictly between 0 and 1 (see Section 2 of \cite{miller-wu-dim}) so there are many opportunities for one of $\eta^1$ or $\eta^2$ to do something that the other does not. On the other hand, it is common for seemingly very different fractal sets to be homeomorphic. For example, if $K_1$ and $K_2$ are compact, non-empty, totally disconnected subsets of $\C$ without isolated points (e.g., Cantor-type sets), then there is a homeomorphism from $\C$ to $\C$ which takes $K_1$ to $K_2$ \cite{moi-geo}.
The main results of this paper show that the topology of SLE$_\kappa$ is random for $\kappa  > 4$. For $\kappa \in (4,8)$, we prove the stronger statement that the topology of the range is random. The results of this paper are in a similar vein to those of \cite{msw-sle-range}, which shows that an $SLE_\kappa$ curve for $\kappa \in (4,8)$ is not determined by its range. Both this paper and \cite{msw-sle-range} answer a seemingly simple question about $SLE$ whose answer is much less obvious than one might initially expect.


\subsection{Summary of results}
\begin{theorem}
\label{thm1}
The topology of chordal $SLE_{\kappa}$ is not deterministic in the following sense: Fix $\kappa \in (4,8),$ and consider two instances of $SLE_{\kappa}$, $\eta^1$ and $\eta^2$ in $\Half$. Then a.s. there is no homeomorphism on $\overline{\Half}$ taking  the range of $\eta^1$ to the range of $\eta^2$.
\end{theorem}
We consider the left and right boundaries of an $SLE$ curve $\eta$ (which are boundary-touching $SLE_{16/\kappa}(\bar{\rho})$ curves). These curves form `bubbles' in $\Half$ (which we characterize explicitly in a later section) which we use as the primary observable to prove \cref{thm1}.

The result also holds for $\kappa \geq 8$. The proof is similar, though a bit more work is needed in the setup.  
\begin{theorem}
\label{thm2}
The topology of chordal $SLE_{\kappa}$ is not deterministic for $\kappa \geq 8$ in the following sense: Consider two  $SLE_{\kappa}$ curves, $\eta^1$ and $\eta^2$ in $\Half$. Then a.s. there is no homeomorphism $\Phi: \overline{\Half} \rightarrow \overline{\Half}$ such that $\Phi(\eta^1)  = \eta^2$ viewed as curves modulo time parametrization. 
\end{theorem}
\begin{remark}
Notice that in \cref{thm2}, we care about parametrized curves, and so  preservation of ranges in this setting makes less sense. Recall $SLE$ in this instance is plane filling. 
\end{remark}
As a natural extension, one can think about the behavior of these curves for varying $\kappa$.

\begin{conjecture}
Let $\kappa_1 , \kappa_2 > 4$ be distinct. Let $\eta_1$ (resp.\ $\eta_2$) be a chordal SLE$_{\kappa_1}$ (resp.\ SLE$_{\kappa_2}$) in $\mathbb H$. Almost surely, there is no homeomorphism $\Phi : \overline{\mathbb H} \rightarrow \overline{\mathbb H}$ such that $\Phi(\eta^1) = \eta^2$ viewed as curves modulo time parametrization. If one of $\kappa_1$ or $\kappa_2$ is in $(4,8)$, a.s., there is no such homeomorphism which takes the range of $\eta_1$ to the range of $\eta_2$.
\end{conjecture}
We expect that this conjecture can be proved using similar ideas to the ones in this paper, but one would have to explicitly compute some of the quantities involved to show that they are $\kappa$-dependent. 
\section{Preliminaries}
Here we discuss a few $SLE$ basics as well as how one defines the more general $SLE_{\kappa}(\bar{\rho})$ processes. Recall that we write $\Half := \{ z \in \C: \mathfrak{Im}(z)> 0\}.$ If $K$ is a bounded closed subset of
$\Half$ such that $\Half \setminus K$ is simply connected, then we call $K$ a hull in $\Half$ w.r.t. $\infty$.
For such $K$, there is a unique $g_K$ that maps $\Half \setminus K$ conformally onto $\Half$ such
that $g_K (z) = z + \dfrac{a}{z} + O\left(\dfrac{1}{z^2}\right)$ as $z \to \infty$. The quantity $a$ is known as the \textit{\textbf{half plane capacity}} of $K$, and is denoted by $\text{hcap}{K}$. It can be shown that $a \geq 0$. The map $g_K$ is said to satisfy the hydrodynamic normalization at infinity. For a real interval $I$, let $\mathcal{C}(I)$ denote the real-valued continuous functions
on I. Suppose $U \in \mathcal{C}([0, T])$ for some $T \in (0,\infty]$. The chordal Loewner equation driven by $U$ is as follows:
\begin{equation}
\label{loewner}
 \dot{g}_t(z) = \dfrac{2}{g_t(z) - U_t} \qquad g_0(z) = z 
 \end{equation}
For $0 \leq t < T$, let $K_t$ and $g_t$ be the  chordal Loewner hulls and maps, respectively, driven by $U_t$.  Suppose that for every $t \in [0, T)$,
\[ \eta_t := \underset{z \in \Half, z \to U_t}{\lim}g_t^{-1}(z) \in \Half \cup \R \]
exists, and $\eta_t, 0 \leq t < T$, is a continuous curve. Then for every $t \in [0, T),
K_t$ is the complement of the unbounded component of $\Half \setminus \eta((0, t])$. We
call $\eta$ the chordal Loewner trace driven by $U_t$. In general, however, such a curve may not exist depending on the choice of driving function. 

An $SLE_{\kappa}$ in $\Half$ from $0$ to $\infty$ is defined by the random family of conformal maps $g_t$ obtained by solving the Loewner ODE driven by Brownian motion. In particular, we let $U_t = \sqrt{\kappa}B_t$, where $B_t$ is a standard Brownian motion. An
$SLE_{\kappa}$ connecting boundary points $x$ and $y$ of an arbitrary simply connected Jordan domain can be constructed as the image of an $SLE_{\kappa}$ on $\Half$ under a
conformal transformation $\Psi \colon \Half \to D$ sending $0$ to $x$ and $\infty$ to $y$. $SLE$ curves are characterized by scale invariance and the domain Markov property, and are viewed modulo reparametrization. The almost sure continuity of the curves of these processes has also been shown in \cite{schramm-sle}.

$SLE(\kappa; \bar{\rho})$, which is often written as $SLE_{\kappa}(\bar{\rho}_L; \bar{\rho}_R)$, is the stochastic process one obtains by solving \eqref{loewner} with a modification on the driving process $U_t$, which we now discuss. It is a natural generalization of $SLE_{\kappa}$ in which one keeps track of additional marked points which are called force points. 
Fix $\bar{x_L} = (x_{l,L}< \dots <
x_{1,L }\leq 0)$ and $\bar{x}_R = (0 \leq x_{1,R} < \dots < x_{r,R})$. We associate with each $x_{i,q}$
for $q \in {L, R}$ a weight $\rho_{i,q} \in \R$. An $SLE_{\kappa}(\bar{\rho}_L;\bar{\rho}_R)$ process with force points $(\bar{x}_L; \bar{x}_R)$
is the measure on continuously growing compact hulls $K_t$ generated by the Loewner chain with $U_t$ replaced by the solution to the system of SDEs given by

\begin{equation}
dU_t = \sum \limits_{i=1}^{l} \dfrac{\rho_{i,L}}{U_t - V^{i,L}} dt +
\sum \limits_{i=1}^{r} \dfrac{\rho_{i,R}}{U_t - V^{i,R}} dt + \sqrt{\kappa}\,dB_t
\end{equation}
\begin{equation}
dV_t^{i, q} = \dfrac{2}{V_t^{i,q}}dt; \quad V_0^{i,q} = x_{i,q}, \quad i \in \N, \quad q \in \{L,R\} 
\end{equation}
The existence and uniqueness of solutions to \eqref{loewner} is discussed in \cite{lsw-restriction}.
In particular, it is shown that there is a unique solution to \eqref{loewner} until the first         time $t$ that $U_t = V_t^{j,q}$ where $\sum \limits_{i=1}^j \rho^{i,q} \leq -2$ for $q \in \{L, R\}$ (we call this time the \textit{\textbf{continuation threshold}}). In particular,if $\sum \limits_{i=1}^j \rho^{i,q} > -2$ for all $1 \leq j \leq |\bar{\rho}^q|$
 for $q \in \{L, R\}$, then \eqref{loewner} has a unique solution for all times $t$. This even holds when
one or both of the $x^{1,q}$ are zero. Note that in \cite{ig1} it is proved that $SLE_{\kappa}(\bar{\rho})$ is generated by a curve and is transient. In this paper, we need only consider the case of two force points, though the main ingredients for the proofs are stated in more generality.   

For $\kappa > 4$, there is also significant interest in the hulls that are generated by the $SLE_{\kappa}$ curves. Duplantier conjectured in \cite{Duplantier_2000, dup-higher-mf} the duality between $SLE_{\kappa}$ and $SLE_{16/\kappa}$, which says that the boundary of an $SLE_{\kappa}$ hull behaves like an $SLE_{16/\kappa}$ curve, for $\kappa >4$. Many versions of this duality have also been shown in \cite{zhan-duality1, zhan-duality2, dubedat-duality, ig1, ig4}. 

Lemma 4.9 in \cite{ig1} asserts that, for $\kappa >4$, the outer boundary $\eta'$ of an $SLE_{\kappa}$ curve is an $SLE_{\kappa}(\bar{\rho})$ process. The $SLE$ curves are realized as \textbf{\textit{flow lines}} of the Gaussian free field (i.e $SLE_{\kappa}(\bar{\rho})$ curves coupled with the Gaussian free field in $\Half$), with the outer boundaries described as \textbf{\textit{counterflow lines}} (in which the coupling is done with the negation of the Gaussian field). Though we do not need this machinery as presented in \cite{ig1} and \cite{miller-wu-dim}, it serves as an excellent framework for
proving some general properties of $SLE_{\kappa}(\bar{\rho})$, some of which we rely on to prove the main results. We state one such fact as follows:

\begin{lemma}
\label{millerwu}
Fix $\kappa > 0$. Suppose that $\eta$ is an $SLE_{\kappa}(\bar{\rho}_L; \bar{\rho}_R)$ process in $\Half$
from $0$ to $\infty$ with force points located at $(\bar{x}_L; \bar{x}_R)$ with $x_{1,L} = 0^-$ and $x_{1,R} = 0^+$ (possibly by taking $\rho_{1,q} = 0$ for $q \in \{L, R \}$). Assume that $\rho_{1,L},\, \rho_{1,R} > -2$. Fix $k \in \N$ such that $\rho = \sum \limits_{j=1}^k \rho_{j,R} \in (\frac{\kappa}{2}-4, \frac{\kappa}{2}-2)$ and  $\epsilon >0$. There exists $p_1 > 0$
depending only on $\kappa,\, \emph{max}_{i,q} |\rho_{i,q}|, \, \rho$, and $\epsilon$ such that if $|x_{2,q}| \geq \epsilon$ for $q \in \{L, R\},
x_{k+1,R}\, - \, x_{k,R} \geq \epsilon$, and $x_{k,R} \leq \epsilon^{-1}$
then the following is true. Suppose that $\gamma$ is a simple curve starting from $0$, terminating in $[x_{k,R}, x_{k+1,R}]$, and otherwise does not hit $\partial \H$. Let $A(\epsilon)$ be the $\epsilon$- neighborhood of $\gamma([0, T])$ and let $$\sigma_1 = \emph{inf} \{t \geq 0 : \eta(t) \in (x_{k,R}, x_{k+1,R})\} \quad \emph{and} \quad \sigma_2 = \emph{inf}\{t \geq 0 : \eta(t) \notin A(\epsilon)\}$$.
Then $\Prob[\sigma_1 < \sigma_2] \geq p_1$.
\end{lemma}
Intuitively, \cref{millerwu} tells us that an $SLE_{\kappa}(\bar{\rho}_L ; \bar{\rho}_R)$ process has a positive chance to stay close to any fixed deterministic curve for a positive amount of time.
\begin{proof}
This is lemma 2.5 in \cite{miller-wu-dim}.
\end{proof}
\section{Proof of Theorem 1}
Consider the left and right boundaries of the $SLE$ curve $\eta$, which are boundary-touching $SLE_{\frac{16}{\kappa}}(\rho)$ curves, with force points starting at $0$. In fact, the left boundary of $SLE_{\kappa}$ turns out to be $SLE_{16/\kappa}(\frac{16}{\kappa} - 4 ; \frac{8}{\kappa} - 2)$ and by symmetry, the right boundary is $SLE_{16/\kappa}(\frac{8}{\kappa} - 2 ; \frac{16}{\kappa} - 4)$. This can be deduced from Proposition 7.31 in \cite{ig1}. These curves are shown in the figure below. The open region between the left and right boundaries has countably many connected components, which are separated by the intersection points of the left and right boundaries, i.e., the cut points of $\eta$. These connected components have a total ordering, and come in four types:
\begin{itemize}
    \item Type 0: Neither the left nor the right boundary of the component intersects the real line.
    \item Type 1: Only the right boundary intersects the real line.
    \item Type 2: Only the left boundary intersects the real line.
    \item Type 3: The left and right boundaries both intersect the real line.
\end{itemize}
\begin{figure}[ht]
\centering
\includegraphics[scale=0.5]{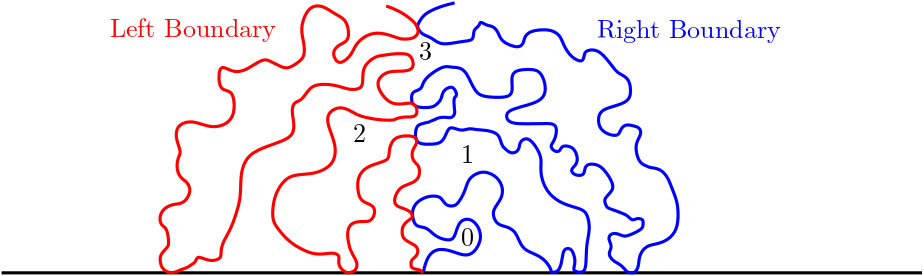}
\caption{We view the complement of the $SLE$ curve as the union of two boundary-touching $SLE_{\kappa}(\bar{\rho})$ processes. We observe `bubbles' of four types, which we use in constructing the observable invariant.}
\end{figure}
Note that $\eta$ is a continuous curve that travels between the positive and negative real axes between any two consecutive components of type 3. This shows that the components of type 3 form a discrete set, to which we may assign a labeling by the integers - written as $$(\dots U_{-1}, U_0, U_1, U_2, \dots )$$ uniquely, modulo index shift. For concreteness, we choose the indexing for the sequence so that $U_0$ is the first type 3 bubble which has Euclidean diameter at least 1. We remark here that our construction relies on a few tail triviality arguments, and so we require the following:
\begin{lemma}
\label{factoid}
Suppose $t > 0$ and let $a_t$ (resp. $b_t$) be the last time before $t$ at which $\eta$ hits the left (resp. right) boundary. Then $\eta|_{[0,t]}$ determines the set of bubbles (i.e. connected components of the region between the left and right boundaries) which are formed before time $\min\{a_t, b_t\}$ as well as their types. 
\end{lemma}
\begin{proof}
This follows trivially from the fact that  $\eta$ cannot cross itself and $\eta( [\min\{a_t , b_t\} , t] )$ disconnects all of the bubbles formed before time $\min\{a_t, b_t\}$ from $\eta(t).$ 
\end{proof}

Between pairs of consecutive type 3 bubbles, $U_{i}$ and $U_{i+1}$, we may either observe a type 1 or 2 bubble, or we may not. Let $E_i$ be the event that there is a type 1 or type 2 bubble between $U_i$ and $U_{i+1}$, and define 
\[ X:= (\dots \mathbbm{1}_{E_{-1}}, \mathbbm{1}_{E_0}, \mathbbm{1}_{E_1}, \mathbbm{1}_{E_2}, \dots ) \]
 the bi-infinite sequence of 0's and 1's consisting of the indicators of the $E_i$'s. 
 
 \begin{lemma}
 \label{seq01}
 For any fixed deterministic bi-infinite sequence of $0$'s and $1$'s $x$, we have $\Prob[X = x] = 0$.
 \end{lemma}
 \begin{proof}
 Consider a left-infinite sequence $y =( \dots y_{-2}, y_{-1}, y_0)$. For $k \in \mathbb{N}$, let $A_k$ be the event that $\{\dots X_{-k-1}, X_{-k} = y \}$. We wish to show that $\Prob[A_0]=0$. We will argue this by contradiction, but we first require a bit of setup. For $r \in \R_{>0}$, $n \in \N$, let $K_r^{(n)}$ be the $n$\nth smallest $k$ such that the Euclidean diameter of $U_k$ is at least $r$. Now, we claim that $\Prob[A_{K_1^{(n)}}] = 0$ for all $n$.  We argue to the contrary, and so we assume that there exists some $n$ such that $\Prob\left [A_{K_1^{(n)}}\right ] >0$. 
 Note that by scale invariance, $\Prob \left [A_{K_r^{(n)}} \right ]$ is independent of $r$, and so depends only on $n$. Consider the event $\bigcap_{i=0}^{\infty} \bigcup_{m \geq i} A_{K_{\frac{1}{m}}^{(n)}}$, which is a tail event for the Brownian motion that drives the $SLE$, for every choice of $n$. To see this, note that \cref{factoid} implies that for each $t$, $\mathcal{F}_t$ determines $A_{K_r^{(n)}}$ for each $r$ which is small enough so that the bubble $U_{K_r^{(n)}}$ is formed before time $\min\{a_t, b_t\}$. Thus, by continuity from above, we note that \[\Prob \left [\bigcap_{i=0}^{\infty} \bigcup_{m \geq i} A_{K_{\frac{1}{m}}^{(n)}} \right ] \geq \Prob \left [A_{K_1^{(n)}} \right] >0 \]
 and so the Blumenthal $0-1$ law implies that, a.s., there exists a sequence $\{r_j\} \rightarrow 0$ such that the events $A_{K_{r_j}^{(n)}}$ occur for all $j$. 
 This implies that there exist infinitely many $k$ such that $A_k$ occurs. Thus, it follows that a.s., $\exists$ infinitely many $k$ such that 
 \[(\dots X_{-k-1}, X_{-k}) = y\] 
 forcing the sequence $y$ to be periodic. We claim that this implies that the sequence $\{X_k\}$ is periodic. 
 Indeed, Let $m$ be the period of $y$. Since there are arbitrarily large $k$ for which $(...X_{-k-1} , X_{-k}) = y$ and $y$ is periodic, it follows that with probability tending to $1$ as $r \rightarrow 0$, the sequence \\$(...X_{-K_r^{(n)}-1} , X_{-K_r^{(n)}})$ is equal to $(...y_{-j-1}, y_{-j})$ for some $j = 1,...,m$. 
 By scale invariance, the probability that this is the case  for all values of $r$ is equal to 1. Thus, as $r \rightarrow \infty$, we see that the entire sequence $\{X_k\}$ is equal to $y$, shifted by some $j = 1,...,m$. This means that if we observe $(...X_{-k-1} , X_{-k})$ for some $k$, we can determine the rest of the sequence $\{X_k\}$, forcing this sequence to be itself periodic.
 
 For $t>0$, we have that by \cref{factoid} $\mathcal{F}_t$ determines the sequence $(\dots X_{-l-1}, X_{-l})$ for some $l$, which by periodicity is enough to determine the sequence $\{X_k\}$. Thus, by \cref{factoid}, $\mathcal{F}_t$ determines $\{X_k\}$ modulo an index shift for each $t>0$, and hence the sequence $\{X_k\}$ is deterministic modulo an index shift.  
 The goal now is to recursively apply \cref{millerwu} to arrive at a contradiction.
 \begin{proposition}
 \label{propzm}
 Let $\mathcal{Z}$ be a finite sequence of $0$'s and $1$'s which does not appear in $y$, with $|\mathcal{Z}| =m$. Then it must hold that \[\Prob\left [\{X_{1}, X_{2}, \dots, X_{m}\} = \{\mathcal{Z}_1, \mathcal{Z}_2, \dots, \mathcal{Z}_{m}\} \right ] >0. \]
 \end{proposition}
 Note that the existence of such a $\mathcal{Z}$ follows from the periodicity of $y$. With this result, we can conclude that the sequence $\{X_k\}$ can contain any finite sequence of $0$'s and $1$'s with positive probability, and hence cannot be periodic and deterministic modulo index shift. We delay the proof of the proposition to state the following key lemma, which uses the fact that the outer boundaries of the curve are $SLE_{\kappa}(\rho_L;\rho_R)$ processes, and more specifically the right boundary, $\eta^R$, conditioned on the left boundary, $\eta^L$, has distribution of  $SLE_{\frac{16}{\kappa}}(-\frac{8}{\kappa};\frac{16}{\kappa}-4)$ (see Lemma 7.1 in \cite{ig1}):
 \begin{figure}
 
     \centering
     \includegraphics[scale=0.48]{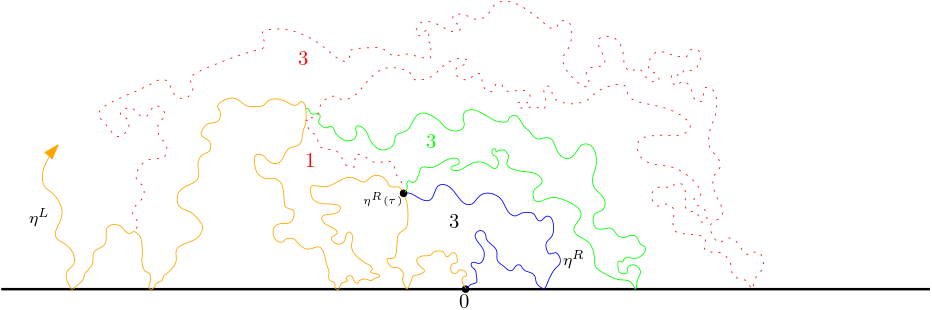}
     \caption{We condition on the left boundary (pictured as the orange curve) and run the right boundary until we first form a type 3 bubble of diameter at least 1 (blue). 
     At this time (denoted $\eta^R(\tau)$), we have two options: either the right boundary hits $[0,\infty)$ before hitting the left boundary again (green), thus forming a type 3 bubble, or it hits the left boundary first (red), forming a type 1 bubble before forming the next type 3 bubble. These events each occur with positive probability.}
     \label{op}
 \end{figure}
 \begin{lemma}
 \label{lemtau}
 Let $\tau$ be a stopping time for $\eta^R$ given $\eta^L$, at which $\eta^R$ forms a type 3 bubble denoted $U_{k_{\tau}}$. Let $E_{k_{\tau}}$ be the event that there is a type 1 or type 2 bubble between $U_{k_{\tau}}$ and $U_{k_{\tau}+1}$, as defined previously.  Then, 
 \[0< \Prob\left[ E_{k_{\tau}}\, \Big | \, \eta^L, \eta^R_{|_{[0, \tau]}} \right] <1. \]
 
 \end{lemma}
 \begin{proof}
 With some setup, this is a straightforward application of \cref{millerwu}. Indeed, let $z_{\tau}:= \eta^R(\tau)$ and define $C_{z_{\tau}}$ to be the connected component of $\eta^L \setminus \R$ containing $z_{\tau}$. Set \[ s^1:= \inf\{ t> \tau: \eta^R \cap [0, \infty) \neq \emptyset\}, \qquad s^2:= \inf\{ t > \tau: \eta^R \cap \eta^L \setminus (C_{z_{\tau}} \cup (-\infty, 0]) \neq \emptyset.\} \]
 By \cref{millerwu}, we have that  \[\Prob \left[s^2 > s^1 \Big | \, \eta^L, \eta^R_{|_{[0, \tau]}}\right] >0; \qquad \Prob \left[s^2 \leq s^1 \Big | \, \eta^L, \eta^R_{|_{[0, \tau]}}\right] >0  \] where the second inequality follows from symmetry considerations. Indeed, we can simply apply \cref{millerwu} to the curve $\eta^R$, under the conditional law given $\eta^L$. In this case, an interval on the left boundary corresponds to a segment of $\eta^L$. Note that these probabilities are strictly less than $1$ as they are both positive and complementary. With this, and appealing to the setting of \cref{op}, we have that $\eta^R[\tau, \infty)$, conditioned on $\eta^L, \eta^R_{|_{[0, \tau]}}$, will either first intersect the left boundary and form a type 1 bubble before forming another type 3 bubble, or it will intersect $[0, \infty)$ before hitting the left boundary again, forming another type 3 bubble. In particular, the event that a type 1 bubble is formed after $U_{k_{\tau}}$ occurs with probability strictly between $0$ and $1$ as desired.
 \end{proof}
 \begin{proof}[Proof of \cref{propzm}]
 We define a sequence of stopping times as follows: For a given bubble $U_i$, let $\tau_i$ be the corresponding time at which $U_i$ is formed. By our choice of indexing of the type 3 bubbles, we have that 
 \[\tau_0:= 1\st \text{ time we form a type 3 bubble of Euclidean diameter at least 1 } \]
 \[\tau_1:= 1\st \text{ time after } \tau_0 \text{ we form a type 3 bubble}\]
 \[\vdots \]
  \[\tau_m:= 1\st \text{ time after } \tau_{m-1} \text{ we form a type 3 bubble.} \]
  Note that $E_{k_{\tau_i}}$is measurable with respect to $\eta^L$ and $\eta^R|_{[0,\tau_{i+1}]}$, and for each $i \in \{1,2, \dots, m\}$, we have that by \cref{lemtau}, 
   \[0<\Prob\left[ E_{k_{\tau_i}}\, \Big | \, \eta^L, \eta^R_{|_{[0, \tau_i]}} \right] <1. \] Thus, it follows that 
   \[\Prob[X_i = \mathcal{Z}_i | X_1 = \mathcal{Z}_1 ,... , X_{i-1} = \mathcal{Z}_{i-1}] > 0 .\] To finish the proof, we note that since $\{X_j = Z_j\}$ is determined by $\eta^L$ and $\eta^R|_{[0,\tau_i]}$ for $i  < j$, so  
   \[\Prob[  X_1 = \mathcal{Z}_1 , ... , X_i  = \mathcal{Z}_i ] = \mathbb{E} \left[ \Prob[ X_i = \mathcal{Z}_i \,|\, \eta^L , \eta^R|_{[0,\tau_i]} ]  \mathbbm{1}_{{ X_1 = \mathcal{Z}_1 , ... , X_{i-1}  = \mathcal{Z}_{i-1}}}   \right].  \] The probability within the expectation on the right hand side is always positive, and so inducting on $i$ (and setting $i=m$ as a final step) yields the desired result.
 \end{proof}
By \cref{propzm}, we see that $\{X_k\}$ can contain any finite sequence of $0$'s and $1$'s not contained in $y$, implying that $\{X_k\}$ cannot be deterministic modulo index shift. This is a contradiction. Thus, $\Prob[A_{K_1^{(n)}}] = 0$ for every $n$. 
 
 Thus, by scale invariance we see that $\Prob[A_{K_r^{(n)}}] = 0$ for every $r$ and $n$. Note that every $k$ is equal to $K_r^{(n)}$ for some rational $r$ and some $n$. Indeed, every $k$\nth bubble has some positive diameter, and there are at most finitely many bubbles before it of larger diameter. Thus, we can set $n$ to be the number of bubbles before the $k$\nth bubble with diameter exceeding that of the $k$\nth bubble, and simply let $r$ be any rational number slightly smaller than this diameter. From this, it follows that 
 \[\Prob[\exists \, k \text{ such that } A_k \text{ occurs }] \leq \Prob \left[ \bigcup_{n \in \N} \bigcup_{r \in \Q_{>0}} A_{K_r}^{(n)} \right] =0.\]
 In particular, we have that $\Prob[A_0]=0$.
 
 \end{proof}
 \begin{proof}[Proof of Theorem]
 Now let $\eta^1$ and $\eta^2$ be two independent $SLE$'s. In order for $\eta^1 \cup \R$ and $\eta^2 \cup \R$ to be homeomorphic via a homeomorphism that takes $\R$ to $\R$, it must be the case that the corresponding bi-infinite sequences $X^1$ and $X^2$ differ by at most an index shift. Indeed, any homeomorphism has to preserve the bi-infinite sequence of connected components lying between the left and right boundaries of the curve, as well as the types of these components. Thus, by the above argument, the probability that $X^1$ is equal to any of the countably many possible index shifted versions of $X^2$ is zero. Hence the probability that $\eta^1 \cup \R$ and $\eta^2 \cup \R$ are homeomorphic, via a homeomorphism that takes $\R$ to $\R$, is 0.
 \end{proof}
 \section{Proof of Theorem 2}
 Here, we require a more subtle argument that relies on a less obvious statistic of observation. In this section, we fix $\kappa \geq 8$ and recall $SLE_{\kappa}$, in this instance, is plane filling. Let $\eta$ be an instance of $SLE_{\kappa}$ in $\overline{\Half}$. We are interested in the successive crossing times (about the origin) of the curve $\eta$, i.e., the times at which $\eta$ hits the real line again, just after having hit it on the opposite side of the origin. We look at one such crossing time, and consider the part of the $SLE$ within this, observing the times it goes back and forth between the boundaries of the crossing excursion. As pictured below in \cref{fig}, these left and right crossings (within the curve) define a sequence of marked points $\{X_k \}$ along the boundary, which accumulate only at the tip of the curve. Via the corresponding Loewner map $g_t^{\eta}$, we may conformally map this configuration as shown in \cref{fig}, so that the tip goes to $0$, and we abtain a sequence of marked points along the left boundary. Notice these marked points are determined by the past, so we can condition on their locations, and the future will still be an $SLE$ by the Markov property.
 
 A bit more care is needed in defining these quantities. Let $\tau(t)$ be the last time before $t$ such that $\eta(t) \in \R$. Define the sets \[
 T_-:= \{t: \eta(\tau(t)) <0\} \qquad T_+:= \{t: \eta(\tau(t)) >0 \} \] and set $\mathcal{S}= \bar{T}_- \cap \bar{T}_+$. Notice that $\mathcal{S}$ is a discrete set since $\eta$ is continuous, and so it cannot cross back and forth between $(-\infty,0)$ and $(0,\infty)$ infinitely many times during any compact time interval contained in $(0,\infty)$. Thus, we may index the elements of $\mathcal{S}$ as a countable sequence of well defined crossing times $\{\tau_j\}$.
 
 Notice that these are not necessarily stopping times (which poses a problem in applying the strong Markov property), but this can be addressed by adopting some notation from the previous section as follows. Let $\eta_j := \eta \restr{[\tau_{j-1}, \tau_j]}$, which is the $j$\nth left-right crossing around $0$ that we observe. For $r>0$, let $J_r^{(n)}$ be the $n$\nth smallest $j$ for which the Euclidean diameter of $\eta_j$ is at least $r$. It is not difficult to see that the set of times $\{\tau_{J_r^{(n)}}\}$ is indeed a set of stopping times. To see this, let $t > 0$. If one sees $\eta|_{[0,t]}$, then one can determine the set $\{\tau_j : j \leq t\}$ (but not necessarily its indexing). 
 This follows from the definition of the times $\{\tau_j\}$ as the intersection points of $T_-$ and $T_+$, as shown previously. Hence $\eta|_{[0,t]}$ determines the set of excursions $\{\eta_j : \tau_j \leq t \}$. We have $\tau_{J_r^{(n)}} \leq t$ if and only if this set of excursions includes at least $n$ elements which have Euclidean diameter at least $r$. Hence $\{\tau_{J_r^{(n)}} \leq t\}$ is determined by $\eta|_{[0,t]}$, which holds for any choice of $t$. 
 
  \begin{figure}[ht]
     \centering
     \includegraphics[scale=0.48]{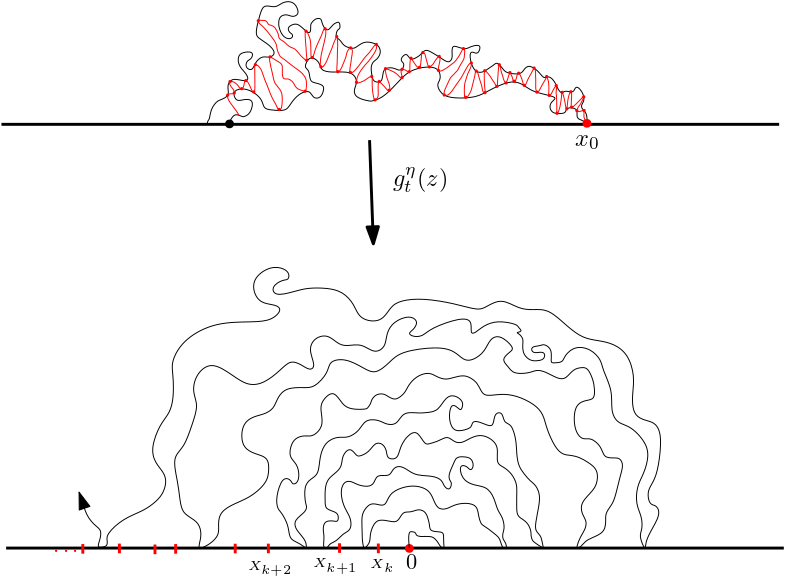}
     \caption{The top picture illustrates a single left-right crossing around $0$, with $x_0= \eta({\tau_J})$ and the corresponding triangulation in red, determined by the (past) piece of the curve making boundary crossings. We conformally map this down, and we consider intervals $[X_{k+1}, X_k]$ in which the tips of future triangles, obtained by left right crossings about $0$, may lie. Some intervals may have multiple, while some may have none.}
     \label{fig}
 \end{figure}

 We fix some $r$ and some $n$, and set $J:= J_r^{(n)}$. Within $\eta_{J}$, i.e. between the outer boundaries of this crossing, we can keep track of the times at which $\{\eta_t; t< \tau_{J}\}$ sequentially hits these boundaries. More precisely, we let $L_{J}$ be the outer boundary of $\eta[0, \tau_{J}]$. We define our sequence of crossing times inductively as follows:
 \[ \sigma_{J, 1}:= \min\{ t > \tau_{J-1}: \eta_t \cap L_J \neq \emptyset\} \]
  \[ \tilde{\sigma}_{J, 1}:= \min\{ t > \sigma_{J, 1}: \eta_t \cap L_{J-1} \neq \emptyset\} \]
  \[ \vdots \]
   \[ \tilde{\sigma}_{J, k}:= \min\{ t > \sigma_{J, k}: \eta_t \cap L_{J-1} \neq \emptyset\} \]
    \[ \sigma_{J, k+1}:= \min\{ t > \tilde{\sigma}_{J, k}: \eta_t \cap L_{J} \neq \emptyset\} \] and so on.  
 The sequences $\{\sigma_{J,k}\}_{k\geq 1}$ and $\{\tilde{\sigma}_{J,k}\}_{k\geq 1}$ define two discrete sets of times that our curve successively hits the outer boundaries $L_J$ and $L_{J-1}$ respectively. We assume without loss of generality that the $J$\nth excursion goes from left to right. By considering only the outer boundary $L_J$ (as a priori $\tau_J$ is a well-defined stopping time), we can construct a sequence of marked points $\{X_{J,k}\}_{k\geq1}$ along the negative real axis, via the (shifted) Loewner map which sends $\eta(\tau_J)$ to $0$. That is to say, $X_{J,k} := g_{\tau_J}(\eta(\sigma_{J,k}))- U_{\tau_J}$. As we are considering a fixed $J$, we may write $X_{J,k}:= X_k$ for ease. 

 We let $N_k= \# \left\{ \text{crossing endpoints of the future curve which lie in the interval } [X_{k+1}, X_k] \right\}.$ In other words, we are looking at $g_{\tau_J}(\eta|_{[\tau_J,\infty)}) - U_{\tau_J}$ as it does these left-right crossings, conditioned on the past, and for each interval we are keeping track of how many endpoints it contains. We wish to show that for every sequence of deterministic integers $\{n_k\}_{k \in \N}$, we have that
 \begin{equation}
 \label{main}
 \Prob[N_k = n_k; \, \forall \, k] =0.
 \end{equation}
 It suffices to show that there are arbitrarily large $k$ such that $\Prob[N_k= n_k]$ is bounded away from $1$. Indeed, the event $\{N_k = n_k \text{ for infinitely many } k\}$ is a tail event for the Brownian motion driving the $SLE$, and the Blumenthal $0-1$ law implies that this has probability $0$ or $1$. Thus, being bounded away from $1$ guarantees that we have \eqref{main}. We do this in cases as follows:
 
 \textit{Case 1}: Assume $\exists$ arbitrarily large $k$ such that $n_k \neq 0$. We claim that $\exists \, q>0$ such that \[ \Prob[N_k =n] \leq 1-q \quad \forall \, n \geq 1.\] To see this, we consider the segment of the curve, just after the $(n-1)^{th}$ crossing is completed. Let $\mathcal{T}_n = \{n^{th} \, \text{time we have a crossing in the interval} \, [X_k, X_{k+1}]\, \}$. Thus $\mathcal{T}_n$ is a stopping time, and conditioned on what we have seen up until this time, the future of the curve is still $SLE$.
 \begin{figure}[b]
     \centering
     \includegraphics[scale=0.48]{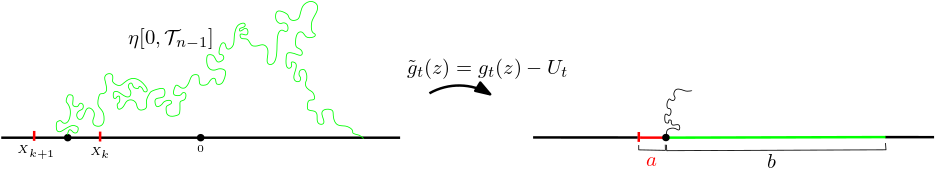}
     \caption{We stop the $SLE$ after it has made its $n-1$\nth crossing in the interval shown. Under the map $\tilde{g}$, we send the tip of the curve to the origin, and analyze the likelihood of either observing two more crossings in the red interval of length $a$, or no more crossings, in which case the interval is swallowed.}
     \label{fig:my_label}
 \end{figure}
 The goal is to have an upper bound on the probability that there are exactly $n$ crossings, and we do so by comparing the harmonic measure (from $\infty$) of the interval $[X_{k+1}, \eta(\mathcal{T}_{n-1})]$, to that of the outer boundary of the curve $\eta[0, \mathcal{T}_{n-1}]$(and more precisely, this is the harmonic measure from $\infty$ in $\Half \setminus \eta[0, \mathcal{T}_{n-1}]$) . These quantities are denoted $a$ and $b$ respectively, as shown in Figure 3. 
 
 The proof relies on the following intuitive argument which we formalize later: If $a$ is larger than $b$, then with positive probability we observe $2$ further crossings, hence $n+1$ total crossings. If $a$ is smaller than $b$, then, with positive probability, we expect the interval $[X_{k+1}, \eta(\mathcal{T}_{n-1})]$ to be covered before we observe the next crossing. In other words, there is always a positive chance that we observe either $n-1$ crossings or $n+1$ crossings, and so \[\Prob[N_k \neq n] >0. \]
 \begin{proposition}
 \label{beff}
 Let $\eta$ be an $SLE_{\kappa}$ from $0$ to $\infty$ in $\Half$ with $\kappa >4$. For marked points $a<0<c$ along the real line, let $E_{a,c}$ be the event that the chordal $SLE_{\kappa}$ trace
visits $[c, \infty)$ before $(-\infty, a]$. Then \[\Prob\left[E_{a,c}\right] = F\left(\dfrac{-a}{c-a}\right) \qquad \text{where } F(x) = \dfrac{1}{Z_{\kappa}}\displaystyle \int_0^x \dfrac{du}{u^{\frac{4}{\kappa}}(1-u)^{\frac{4}{\kappa}}}\] and $Z_{\kappa}$ is chosen so that $F(1)=1$.
 \end{proposition}
 \begin{proof}
 This is Theorem 10 in \cite{beffara-sle}.
 \end{proof}
 \begin{remark}
 \label{rem2}
 It is possible to get an estimate which is weaker than Theorem 3 above, but which is still sufficient for our purposes, via the following elementary argument. For $x\in \R$, let $t_x := \inf [t \geq 0: \eta(t)= x].$ If we let $P(n) = \Prob[t_n < t_{-1}]$, a bit of thought shows that \[P(n) \geq P(n-1)[1- P(n)] \] which thus implies that \[P(n) \geq \dfrac{P(n-1)}{1+ P(n-1)}.\] The equality case can be realized as $P(n) = \dfrac{1}{n+1}$, the details of which we omit. By considering $f(x) = \dfrac{x}{x+1}$, which is increasing on $\R_{\geq 0}$, we find that \[P(n) \geq f(P(n-1)) \geq f^{(2)}(P(n-2) \dots \geq f^{(n)}\left(\frac{1}{2}\right) = \dfrac{1}{n+1}\] which gives a rough (yet easy to compute) estimate. Note, for our purposes, we only require a positive probability.
 \end{remark}
We return to the notation introduced in \cref{fig:my_label}, and we consider the the behavior of the $SLE$ curve given the relative quantities $a$ and $b$. In particular, we require the following two key lemmas to prove the original claim: 
\begin{lemma}
If $a\leq b$, it holds with conditional probability at least $\frac{1}{2}$, given $\eta|_{[0,\mathcal{T}_{n-1}]}$, that $\eta|_{[\mathcal{T}_{n-1},\infty)}$ hits $X_{k+1}$ before $[b,\infty).$
\end{lemma}
\begin{proof}
Notice that by symmetry, there is a positive chance that we disconnect $[X_{k+1}, \eta(\mathcal{T}_{n-1})]$ before hitting $b$. Indeed, this follows from the fact that $\Prob[t_{-1} < t_1] = \frac{1}{2}.$ 
\end{proof}
\begin{lemma}
There exists a deterministic $\kappa$-dependent constant $c > 0$ such that if $a > b$,  it holds with conditional probability at least $c$ given $\eta|_{[0,\mathcal{T}_{n-1}]}$ that $\eta|_{[\mathcal{T}_{n-1} , \infty)}$ crosses between $(-\infty,0)$ and $(0,\infty)$ at least twice before hitting $X_{k+1}.$
\end{lemma}
\begin{proof}
If $a>b$, then we can apply the estimate given in \cref{beff} via a two step process. We retain the notation from \cref{rem2}, and define $t_{-a}$ and $t_b$ as discussed, after having mapped $\eta([0,\mathcal{T}_{n-1}])$ to the real line via the map $\tilde{g}$. Note that \cref{beff} implies that $\exists \, p>0$ such that $\Prob[t_b < t_{-a/2}] \geq p$. In fact, we have assumed $a>b$, so $p$ in this instance can be thought of as a universal bound. We condition on this event occurring, and we look at the harmonic measure of the outer boundary curve $\tilde{\eta}$ of this most recent crossing. Note that $\text{hm}_{\Half \setminus \tilde{\eta}}(\infty, \tilde{\eta})$ is bounded above by the harmonic measure of the outer boundary at the time we hit $-\dfrac{a}{2}$. This follows from the fact that the harmonic measure can only increase, as we observe more of the curve. Moreover, the law of this harmonic measure, divided by $a$, is independent of $a$ by scale invariance, and is almost surely finite. This implies that $\exists \, C = C(p) >0$ such that \[ \Prob\left [\text{hm}_{\Half \setminus \tilde{\eta}} ( \infty,  \tilde{\eta} ) \leq Ca\right ] \geq 1- \dfrac{p}{2}\]
from which it follows that 
\[\Prob\left[\text{hm}_{\Half \setminus \tilde{\eta}}(\infty, \tilde{\eta}) \leq Ca, \, t_b < t_{-a/2}\right] \geq \dfrac{p}{2}. \]
This bound guarantees a positive probability that, after we have observed the first crossing, the harmonic measure of the outer boundary is not too large. Now we condition on this event, and we apply Proposition 4.1 to the quantities $Ca$ and $\dfrac{a}{2}$. In particular, This yields a positive $\kappa$-dependent constant lower bound for the probability that $\eta|_{[\mathcal{T}_{n-1} , \infty)}$ has at least two crossings before hitting $X_{k+1}$.
\end{proof}
\textit{Case 2}: $n_k = 0$ for all but finitely many $k$.

This condition implies that the $SLE$ travels a positive distance of time without any left-right crossings, which happens with probability $0$. This shows that for any fixed deterministic sequence $\{n_k \}_{k \in \N}$ with only finitely many non-zero elements, we have that 
\[ \Prob[ \{N_k\} = \{n_k\}] =0. \]
\begin{figure}[ht]
    \centering
    \includegraphics[scale=0.48]{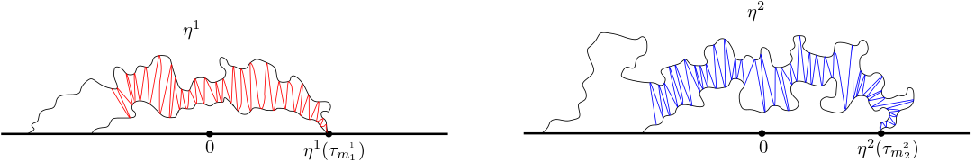}
    \caption{We observe two instances of $SLE$, $\eta^1$ and $\eta^2$, stopped after the $m_1$\nth and $m_2$\nth crossings respectively. Any homeomorphism between the two should send one tip to the other, and retain the structure of the future crossings (i.e., preserve the corresponding sequences $\{N_k^j\}$).}
    \label{homeo}
\end{figure}
\begin{proof}[Proof of theorem]
Consider two instances of $SLE_{\kappa}$ in $\Half$, $\eta^1$ and $\eta^2$, with corresponding sequences of points $\{X_{m_1,k}^1\}_{k \in \N}$ and $\{X_{m_2,k}^2\}_{k \in \N}$ respectively, for fixed indicies $m_1, m_2 \in \mathcal{S}$, corresponding to the $m_1$\nth crossing of $\eta^1$ and $m_2$\nth crossing of $\eta^2$ respectively. Here, we indicate objects associated with $\eta^j$ for $j \in \{1,2\}$ by a superscript $j$. Note that by construction, $m_1 = J_{r_1}^{(n_1),1}$ and $m_2 = J_{r_2}^{(n_2),2}$ for some $n_1, n_2$ and (rational) $r_1, r_2$. Each sequence of points $\{X_k^j\}_{k \in \N}$ generates a sequence $\{N_k^j\}_{k \in \N}$ for $j \in \{1,2\}$ and so by the independence of $\eta^1$ and $\eta^2$, as well as \eqref{main}, we have that for any choice of $m_1, m_2$ and number $l$  
\[\Prob\left[N_k^1 = N_{k+l}^2; \,\, \forall \, k\right] = \Prob\left [N_k^1 = N_{k+l}^2; \,\, \forall \, k \,|\, \eta^2\right ] =0. \] This implies that 
\begin{equation}
\label{shift}
\Prob\left[\exists \, l \text{ s.t } N_k^1 = N_{k+l}^2; \, \forall \, k \right] =0  
\end{equation}
as there are countably many possible choices of $l$, meaning we can apply this very argument for each fixed choice of $l$, and apply the union bound. 

Observe that a homeomorphism  from $\overline{\Half}$ to itself taking $\eta^1$ to $\eta^2$, modulo time parametrization, must preserve the number of left right crossings of the `future' curves, which correspond to the sequences $\{N_k^j\}$, and it must take $\eta^1(\tau_{m_1}^1) $ to $\eta^2(\tau_{m_2}^2) $ for some $m_2$. In particular, as in the setting of Figure 4, for any fixed $m_1$ and $m_2$ there is no homeomorphism which takes $\eta^1$ to $\eta^2$ and $\eta^1(\tau_{m_1}^1)$ to $\eta^2(\tau_{m_2}^2) $ by \eqref{shift}. As the set $\mathcal{S}$ of crossing times is discrete, this holds for any choice of indices $m_1$and $m_2$ , where there are only countably many choices. Thus, it must hold that, 
\[\Prob\left[ \exists \, \text{ a homeomorphism } \Phi: \overline{\Half} \to \overline{\Half} \text{ taking } \eta^1 \text{ to } \eta^2  \right] =0. \]

\end{proof}
\section*{Acknowledgements}
I would like to thank Prof. Ewain Gwynne for suggesting this problem, and for answering my many questions about the material presented here, as well as $SLE$ in general. I would also like to thank Prof. Gregory Lawler for suggesting readings and proof techniques to supplement this paper.
\bibliographystyle{plain}
\bibliography{cibib}
\end{document}